\documentclass[11pt]{amsart}
\usepackage{amsfonts}
\usepackage{latexsym}
\usepackage{amscd}
\usepackage{amsmath}
\usepackage{amssymb}
\usepackage{graphicx}
\usepackage{xypic}
\usepackage[mathscr]{eucal}
\usepackage[dvips]{color}

\def\vv{{\underline{\nu}}}

\def\1{{\underline{1}}}

\def\E{\mathcal{E}}
\def\card{\rm{card}}

\newtheorem{theo}{Theorem}[section]
\newtheorem{prop}{Proposition}[section]
\newtheorem{lem}{Lemma}[section]

\theoremstyle{definition}
\newtheorem{de}{Definition}[section]

\title[Poincar\'e series of  multiplier ideals]{\rm{The
Poincar\'e series of multiplier ideals of a simple complete ideal in
a local ring of a smooth surface}}
\author{C.~Galindo \and F.~Monserrat}
\curraddr{Departament de Matem\`{a}tiques, Universitat Jaume I,
Campus de Riu Sec. s/n, 12071 Castell\'{o} (Spain)} \email{
galindo@mat.uji.es} \curraddr{Instituto Universitario de
Matem\'atica Pura y Aplicada, Universidad Polit\'ecnica de Valencia,
Camino de Vera s/n, 46022 Valencia (Spain)}
\email{framonde@mat.upv.es}
\date{}
\thanks{Supported by Spain Ministry of Education
 MTM2007-64704, JCyL VA025A07 and Bancaixa P1-1A2005-08}

\begin{document}
\maketitle

\begin{abstract}
For a simple complete ideal $\wp$ of a local ring at a closed point
on a smooth complex algebraic surface, we introduce an algebraic
object, named Poincar\'e series $P_{\wp}$, that gathers in an
unified way the jumping numbers and the dimensions of the vector
space quotients given by consecutive multiplier ideals attached to
$\wp$. This paper is devoted to prove that $P_{\wp}$ is a rational
function giving an explicit expression for it.
\end{abstract}

\section{Introduction}
Multiplier ideals are a recent and important tool in singularity
theory and in birational geometry. They have the virtue of giving
information on the type of singularity corresponding to an ideal,
divisor or metric and of accomplishing several vanishing theorems
which made them very useful. As a reference, including historic
development, for this concept we refer to \cite[Ch. 9, 10, 11]{la}.
In spite of the utility of multiplier ideals, which is due to that
many of their properties and applications are known, to compute
these ideals is very hard because it involves facts as either to
calculate resolution of singularities or to obtain very difficult
integrals. As a consequence, very few explicit computations are
known. The most remarkable is the one of multiplier ideals of
arbitrary monomial ideals \cite{ho}.

Intimately related to multiplier ideals are the  jumping numbers
(see \cite{e-l-s-v}, where one can also read about the antecedents
of these numbers). Jumping numbers are a sequence of rational
numbers that provide a sequence of invariants for the singularity in
question, extending in a natural way the information given by the
log-canonical threshold since this is the smallest jumping number.

In the line of looking for explicit computations related to
multiplier ideals, we shall consider the local ring $R$ at a closed
point on a smooth complex algebraic surface. Our aim consists of
studying the sequence of multiplier ideals of a simple complete
ideal $\wp$ of $R$. It is well known that the class of simple
complete ideals plays a crucial role in the so-called Zariski theory
of complete ideals \cite{z-1,zar}. This theory was inspired by the
work of Enriques and Chisini \cite[L. IV, Ch. II, Sect. 17]{e-f} and
it has had further developments due mainly to  Lipman (see
\cite{li1}) who also gave a concept preceding the one of multiplier
ideal, the so-called adjoint ideal \cite{li2}.

Very recently, J\"{a}rvilehto  \cite{ja} obtained  an explicit
description of the jumping numbers attached to simple complete
ideals $\wp$ as above. He gives a formula where the set of jumping
numbers $ \mathcal{H}$ can be seen as a union of finitely many sets
$\mathcal{H}=\bigcup_{i=1}^{g^*+1} \mathcal{H}_i$ and each
$\mathcal{H}_i$ is determined by the maximal contact values of the
divisorial valuation defined by $\wp$ \cite{zar}. Furthermore the
jumping numbers of an ideal in the local ring at a rational
singularity on a complex algebraic surface can also be obtained by
an algorithm provided by Tucker in \cite{tuc}.

In this paper, we consider the family (ordered by inclusion) of
multiplier ideals defined by $\wp$ and taking into account that the
vector space given by the quotient between two consecutive
multiplier ideals is finitely generated (a consequence of Nakayama
Lemma), we attach to $\wp$ a Poincar\'e series whose coefficients
are the dimensions of the above vector spaces (see Definition
\ref{poi}). With the help of the explicit description  of the
jumping numbers in \cite{ja}, we give in Theorem \ref{DOS} a
characterization of the jumping numbers belonging to each set
$\mathcal{H}_i $, that depends on the fact that certain irreducible
exceptional divisors of a log-resolution of $\wp$ {\it contribute}
these jumping numbers. This concept was introduced in \cite{s-t} by
Smith and Thompson, where the set of irreducible exceptional
divisors which contribute jumping numbers associated with a singular
curve on a smooth surface is described. A similar result was
obtained by Favre and Jonsson using different techniques (see
Proposition 2.4, Lemma 2.11 and Fact 2 in the proof of Theorem 6.1
of \cite{fav}). These contributing exceptional divisors are
essential for our development and allow to prove our main result
(Theorem \ref{UNO}) which states that the mentioned Poincar\'e
series is a rational function and provides an explicit computation
for it. This series is an algebraic object that involves jumping
numbers and the dimensions of its above mentioned corresponding
vector spaces. The explicit description we give allows to get
information  that multiplier ideals add to the jumping numbers. In
fact, we prove that the coefficient corresponding to each jumping
number $\iota$ is the sum of the dimensions of certain vector spaces
attached to the indices $i$ such that $\iota \in \mathcal{H}_i$.
These dimensions are always one except for the last index $g^*+1$,
in which case they can be calculated from the expression of $\iota$
described in \cite{ja}. An important aid to compute our Poincar\'e
series is the description we show in Theorem \ref{TRES} of the
previous multiplier ideal to a given one.

To make easier the reading of this paper, in the next section we
state the necessary notations and our main results while the proofs
are relegated to the last section.

\section{Results}\label{results}

We fix, along the paper, a local ring $R$ at a closed point of a
smooth complex algebraic surface. Denote by $K$ the quotient field
of $R$. Consider a simple complete ideal $\wp$ of $R$ and set $\nu$
its corresponding valuation (of $K$ centered at $R$). $\nu$ is
defined by a divisor $E_n$ obtained from a finite simple sequence of
point blowing ups
\begin{equation}
\label{spip} \pi:   X = X_{n} \stackrel{\pi_{n}}{\longrightarrow}
X_{n-1} \longrightarrow \cdots \longrightarrow X_{1}
\stackrel{\pi_{1}}{\longrightarrow} X_{0}= {\rm Spec}(R),
\end{equation}
determined by the centers of $\nu$ at the spaces $X_j$ \cite{zar}.
We shall denote by $E_j$ the prime exceptional divisor created by
$\pi_j$ (and, abusing of notation, also its strict transform on
$X$). Associated with the above objects, there exists a rooted tree,
$\Gamma$, usually named the {\it dual graph} of $\nu$ (or of $\wp$),
where each vertex represents an exceptional divisor $E_j$ (on $X$)
and two vertices are joined by an edge whenever the corresponding
divisors intersect (see Figure
 \ref{fig0}); its root is the vertex corresponding to the first exceptional divisor, $E_1$.
 The {\it star vertices} of the dual graph (labelled with $st_i$ in Figure
\ref{fig0}) will be those whose associated exceptional divisors
$E_{st_i}$ meet three distinct prime exceptional divisors. From now
on, we shall denote by $g^*$ the number of star vertices. A vertex
of $\Gamma$ will be called a {\it dead vertex} if it has only one
adjacent vertex.

 \begin{figure}[h]
$$
\unitlength=1.00mm
\begin{picture}(80.00,30.00)(-10,3)
\thicklines \put(-5,30){\line(1,0){41}} \put(44,30){\line(1,0){31}}
\put(38,30){\circle*{0.5}} \put(40,30){\circle*{0.5}}
\put(42,30){\circle*{0.5}} \put(30,10){\line(0,1){20}}
\put(50,20){\line(0,1){10}} \put(60,0){\line(0,1){30}}
\put(10,15){\line(0,1){15}} \thinlines \put(20,30){\circle*{1}}
\put(30,30){\circle*{1}} \put(50,30){\circle*{1}}
\put(60,30){\circle*{1}} \put(65,30){\circle*{1}}
\put(70,30){\circle*{1}} \put(75,30){\circle*{1}}
\put(75,30){\circle{1}}

\put(30,20){\circle*{1}} \put(60,20){\circle*{1}}
\put(60,10){\circle*{1}} \put(10,30){\circle*{1}}
\put(30,10){\circle*{1}} \put(50,20){\circle*{1}}
\put(60,0){\circle*{1}} \put(-5,30){\circle*{1}}
\put(0,30){\circle*{1}} \put(5,30){\circle*{1}}
\put(15,30){\circle*{1}} \put(25,30){\circle*{1}}
\put(35,30){\circle*{1}} \put(45,30){\circle*{1}}
\put(55,30){\circle*{1}} \put(10,25){\circle*{1}}
\put(10,20){\circle*{1}} \put(10,15){\circle*{1}}
\put(30,25){\circle*{1}} \put(30,15){\circle*{1}}
\put(35,30){\circle*{1}}
\put(4.5,19){$\Gamma_1$}
\put(24.5,14){$\Gamma_2$}
\put(54.5,4){$\Gamma_{g}$} \put(9,32){$st_1$} \put(29,32){$st_2$}
\put(57.5,32){$st_{g}$}
\put(70.5,25){$\Gamma_{g+1}$}
\end{picture}
$$
\caption{The dual graph of a divisorial valuation.} \label{fig0}
\end{figure}
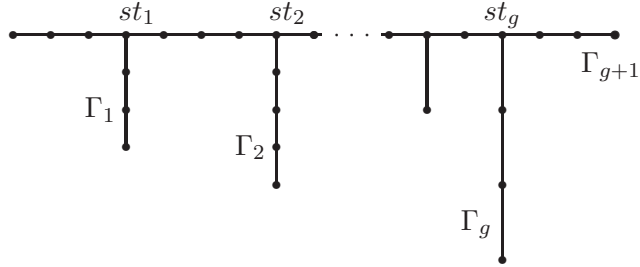

An exceptional divisor $E_{j_0}$ {\it precedes} another one
$E_{j_1}$ if $j_0<j_1$. Also, $E_{j_1}$ is named {\it proximate} to
$E_{j_0}$ whenever $E_{j_0}$ precedes $E_{j_1}$ and the point to be
blown-up in $X_{j_{1}-1}$  to create $E_{j_1}$ is in the strict
transform of $E_{j_0}$. If $E_{j_1}$ is proximate to, at most, one
prime exceptional divisor, then we shall say that $E_{j_1}$ is a {\it
free} divisor; otherwise $E_{j_1}$ will be a {\it satellite}
divisor. Notice that the divisors corresponding to star vertices,
$E_{st_i}$, are characterized by the fact that $E_{st_i}$ is
satellite and $E_{st_i+1}$ is free.

We associate to each star vertex $st_i$, inductively, a rooted
subtree $\Gamma_i$ of $\Gamma$ in the following manner: $\Gamma_1$
is the subgraph of $\Gamma$ whose vertices are those corresponding
to the divisors $E_j$ such that $j\leq st_1$ and, for $1<i\leq g^*$,
$\Gamma_i$ is the subgraph of $\Gamma$ whose vertices correspond to
divisors $E_j$ such that $j\leq st_i$ but they are not vertices of
$\Gamma_{k}$, $1\leq k\leq i-1$; the root of $\Gamma_1$ is the one
of $\Gamma$ and, for each $i>1$, the root of $\Gamma_i$ is the
vertex adjacent to $st_{i-1}$. Also we define $\Gamma_{g^*+1}$ to be
the rooted subtree of $\Gamma$ whose vertices are those which are
not in $\Gamma_{k}$, $1\leq k\leq g^*$; its root is the vertex
adjacent to $st_{g^*}$.

Along this paper we stand $\{\bar{\beta}_i\}_{i=0}^{g+1}$ for the {\it
maximal contact values} (or {\it Zariski exponents}) of the
valuation $\nu$ \cite[Sect. 6]{spi}. Also, set $e_i : = \gcd
(\bar{\beta}_0, \bar{\beta}_1, \ldots, \bar{\beta}_i)$, $0 \leq i
\leq g$ and $n_i := e_{i-1}/e_i$, for $1 \leq i \leq g$. If the last
prime exceptional divisor $E_n$ is free (as in Figure \ref{fig0})
then $g=g^*$ and, otherwise, $g=g^*+1$ (in this case there is no
subgraph $\Gamma_{g+1}$). Also we shall denote by $F_i$ $(1 \leq i
\leq g^*)$ the divisor $E_{st_i}$ corresponding to a star vertex of
$\Gamma$ and we stand $F_{g^* +1}$ for the last obtained exceptional
divisor $E_n$.

A concept that we shall use often in this paper is given in the
following

\begin{de}\label{general}
Given a prime exceptional divisor $E_j$, an {\it $E_j$-general
element} for the valuation $\nu$ will be an element $\varphi\in R$
giving an equation of an analytically irreducible germ of curve
whose strict transform on $X_j$ is smooth and intersects $E_j$
transversally at a non-singular point of the exceptional locus. The
$E_n$-general elements are usually named {\it general elements of the
valuation} $\nu$.
\end{de}

A remarkable fact is the description of the maximal contact values
of $\nu$ as values of certain $E_j$-general elements.
Specifically, if $\varphi_j \in R$ denotes any $E_j$-general element
for $\nu$, then $\bar{\beta}_i=\nu(\varphi_{l_i})$, $0\leq i\leq
g+1$, where $l_0<l_1<\ldots <l_{g}$ are the subindexes of the
divisors $E_{l_i}$ corresponding to the first $g+1$ dead vertices of
$\Gamma$ and $l_{g+1}=n$.

Our goal in this paper is to define and compute an object containing
information concerning the multiplier ideals attached to the ideal
$\wp$. To do it, since the sequence $\pi$ in (\ref{spip}) gives a
log-resolution of the ideal $\wp$, we can consider the effective
divisor $D = \sum_{j=1}^n a_j E_j$ such that $\wp \mathcal{O}_{X} =
\mathcal{O}_{X}(-D)$. Notice that if $\varphi_j \in R$ is an
$E_j$-general element for $\nu$ then $a_j = \nu (\varphi_j)$. Thus,
for any positive rational number $\iota$, the {\it multiplier ideal}
of $\wp$ and $\iota$ can be defined as $\mathcal{J} (\wp^\iota) :=
\pi_* \mathcal{O}_X ( K_{X|X_0} - \lfloor \iota D \rfloor )$, where
$K_{X|X_0}$ is the relative canonical divisor and $\lfloor \cdot
\rfloor$ represents the round-down or the integral part of the
corresponding divisor. The family of multiplier ideals is totally
ordered by inclusion, parameterized by non-negative rational
numbers. Furthermore, there is an increasing sequence $\iota_0 <
\iota_1 < \cdots$ of positive rational numbers, called {\it jumping
numbers}, such that $\mathcal{J} (\wp^\iota) = \mathcal{J}
(\wp^{\iota_l})$ for $\iota_l \leq \iota < \iota_{l+1}$ and
$\mathcal{J} (\wp^{\iota_{l+1}}) \varsubsetneq \mathcal{J}
(\wp^{\iota_l})$ for each $l \geq 0$; $\iota_0$, usually named
log-canonical threshold of $\wp$, is the least positive rational
number with such a property.

Denote, as above, by $g^*$ the number of star vertices in $\Gamma$
and \label{FOUR} set
\[
\mathcal{H}_i : = \left \{ \iota(i,p,q,r):= \frac{p}{e_{i-1}}+
\frac{q}{\bar{\beta}_i}+ \frac{r}{e_i} \; \mid \;
\frac{p}{e_{i-1}}+ \frac{q}{\bar{\beta}_i} \leq \frac{1}{e_i}; p,
q \geq 1, r \geq 0 \right \}
\]
whenever $1 \leq i \leq g^*$, and
\[
\mathcal{H}_{g^* +1} : = \left \{ \iota(g^* +1 ,p,q):=
\frac{p}{e_{g^*}}+ \frac{q}{\bar{\beta}_{g^* +1}} \; \mid \;  p, q
\geq 1 \right \},
\]
$p$, $q$ and $r$ being integer numbers. In \cite{ja}, it is proved
that the set $\mathcal{H}$ of jumping numbers of the ideal $\wp$ can
be computed as $\mathcal{H}= \cup_{i=1}^{g^*+1} \mathcal{H}_i$.

Now, inspired by the terminology introduced in \cite{s-t}, for a
simple complete ideal $\wp$ of $R$ and its corresponding
log-resolution $\pi$ (see (\ref{spip})) and divisor $D =
\sum_{j=1}^n a_j E_j$, we give the following definition:

\begin{de}
\label{contr}{\rm  A {\it candidate jumping number} from  a prime exceptional divisor
$E_j$ given by $\pi$ is a positive rational number  $\iota$ such that $\iota a_j$ is an integer number.
Also, we shall say that $E_j$ {\it
contributes} $\iota$ whenever $\iota$ is a candidate jumping number from $E_j$
and $\mathcal{J} (\wp^{\iota}) \subsetneq \pi_* \mathcal{O}_X ( -
\lfloor \iota D \rfloor + K_{X|X_0} + E_j)$.  }
\end{de}

Assume $\iota \in \mathcal{H}$ and $\iota \neq \iota_0 = \min
\mathcal{H}$. We denote by $\iota^<$ the largest jumping number
which is less than $\iota$. By convention we set $\mathcal{J}
(\wp^{\iota_0^<}) =R$.
 Nakayama Lemma proves that, for any $\iota \in \mathcal{H}$, $\mathcal{J} (\wp^{\iota^<}) / \mathcal{J} (\wp^{\iota})
$ is a finitely generated $\mathbb{C}$-vector space, $\mathbb{C}$
being the field of complex numbers. Thus, we can define the object
to be studied as

\begin{de}
\label{poi} {\rm Let $\wp$ a simple complete ideal of $R$. The {\it
Poincar\'e series of multiplier ideals of $\wp$} is defined to be
the following fractional series:
\[
P_{\wp} (t) := \sum_{\iota \in \mathcal{H}} \dim_{\mathbb{C}} \left(
\frac{\mathcal{J} (\wp^{\iota^<})}{\mathcal{J} (\wp^{\iota})}
\right) t^\iota,
\] }
$t$ being an indeterminate.
\end{de}

Our main result is

\begin{theo}
\label{UNO} The Poincar\'e series $P_{\wp} (t)$ can be expressed as
\[
P_{\wp} (t)=\frac{1}{1-t} \sum_{i=1}^{g^*} \sum_{\iota \in
\mathcal{H}_i, \iota < 1} t^\iota + \left( \frac{1}{1-t} + \frac{t}{(1-t)^2} \right) \sum_{\iota \in \Omega} t^\iota,
\]
where $$\Omega:=\{\iota\in \mathcal{H}_{g^* +1}\mid  \iota\leq 2
\mbox{ and } \iota-1\not\in \mathcal{H}_{g^* +1} \}.$$
\end{theo}

We must clarify that $P_{\wp} (t)$ is not an element in
$\mathbb{C}(t)$ but there exist finitely many (exactly $g^* +1$)
``roots of $t$" which could be considered as another indeterminates,
say $z_1, z_2, \ldots, z_{g^* +1}$, such that $P_{\wp} (t) \in
\mathbb{C}(z_1, z_2, \ldots, z_{g^* +1})$. To prove this theorem  we
shall use the following results:

\begin{theo}
\label{DOS} A jumping number $\iota$ of a simple complete ideal $\wp$ belongs to the set
$\mathcal{H}_i$ ($1 \leq i \leq g^*+1$) if and only if the prime
exceptional divisor $F_i := E_{st_i}$ contributes $\iota$.
\end{theo}

\begin{theo}
\label{TRES} Let $\iota$  be a jumping number of a  simple complete
ideal $\wp$. Then
\[
\pi_*\mathcal{O}_X \left( - \lfloor \iota D \rfloor + K_{X|X_0} +
\sum_{l=1}^s F_{i_l} \right) = \mathcal{J} \left(\wp^{\iota^<}
\right),
\]
where $\{i_1,i_2,\ldots,i_s\}$ is the set of indexes $i$, $1 \leq i \leq g^* +1$, such that
$\iota\in \mathcal{H}_{i}$.
\end{theo}

As a direct consequence of Theorem \ref{DOS} and Clause (c) of the
forthcoming Proposition \ref{CINCO}, we get the following result
that tells us which are the prime exceptional divisors contributing
a given jumping number.

\begin{prop}\label{iuiu}
The prime exceptional divisors that contribute a jumping number
$\iota$ of a simple complete ideal $\wp$ are those divisors $F_i$
such that $\iota\in \mathcal{H}_i$.
\end{prop}

\noindent {\it Remark}. In \cite{tuc} it is announced that, in a
future work of the author, a similar result to Proposition
\ref{iuiu} for jumping numbers that are less than one will be
provided.

\section{Proofs}
Along this section we shall use the above notations. We start by proving Theorem \ref{DOS}.
\subsection{Proof of Theorem \ref{DOS}}
It will be useful the following result, whose proof can be deduced from Section 3 of \cite{s-t} and,
therefore, we  omit it.
\begin{prop}
\label{CINCO} Let $\iota$ be a positive rational number and $E_j$
a prime exceptional divisor given by the sequence $\pi$ of
(\ref{spip}). Then
\begin{itemize}
\item[(a)] $\pi_* \mathcal{O}_X (-\lfloor \iota D \rfloor +
K_{X|X_0})\not=\pi_* \mathcal{O}_X (-\lfloor \iota D \rfloor +
K_{X|X_0}+E_j)$ if and only if $-\lfloor \iota D \rfloor \cdot
E_j\geq 2$.

\item[(b)] Assume that $\iota$ is a jumping number. Then $E_j$
contributes $\iota$ if and only if $\iota$ is a candidate jumping number from
$E_j$ and $-\lfloor \iota D \rfloor \cdot E_j\geq 2$.

\item[(c)] If $E_j$ contributes a jumping number $\iota$ then
$E_j=F_i$ for some $i \in \{1,2,\ldots,g^*+1\}$. $\Box$

\end{itemize}

\end{prop}

We shall divide the proof of the direct implication of the theorem
in two parts, 1 and 2.
\begin{itemize}
\item[{\bf 1.}] Assume that $g^* = g$, that is $\Gamma$ contains a
subgraph $\Gamma_{g+1}$, and consider three subcases.
\end{itemize}

\begin{itemize}
    \item[{\bf a.}] Let us prove that the divisor $F_g$ contributes  any
    $\iota \in {\mathcal H}_g$.
\end{itemize}

Consider the subgraph of $\Gamma$ whose vertices are those
corresponding to $F_g$ and the prime exceptional divisors that
meet $F_g$, that we denote by $F'_{g-2}$, $F'_{g-1}$ and
$F'_{g+1}$ (see Figure \ref{fig1}). We suppose that $F'_{g-1}$ is
the exceptional divisor created immediately before that $F_g$
(that is, $E_{st_g-1}$). Notice that, in Figure \ref{fig1},
$F'_{g-2}$ and $F'_{g-1}$ can appear interchanged but our
reasoning in that case will be similar.

\begin{figure}[h]
$$
\unitlength=1.00mm
\begin{picture}(80.00,30.00)(0,0)
\thicklines \put(30,20){\line(1,0){20}} \put(40,11){\line(0,1){8}}
\put(30,20){\circle*{1}} \put(40,20){\circle*{1}}
\put(50,20){\circle*{1}} \put(40,10){\circle*{1}}
\put(28,22){$F'_{g-2}$} \put(38,22){$F_{g}$} \put(48,22){$F'_{g+1}$}
\put(42,10){$F'_{g-1}$} \put(22,19){$\cdots$} \put(54,19){$\cdots$}
\put(40,4){$\vdots$}
\end{picture}
$$
\caption{$F_g$ contributes ${\mathcal H}_g$} \label{fig1}
\end{figure}
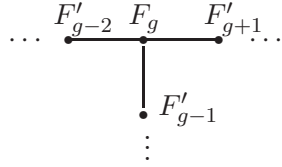

Recall that $E_j \cdot E_j = -1 - \card (\mathcal{P}_j)$,
$\mathcal{P}_j$ being the set of prime exceptional divisors which
are proximate to $E_j$ and $\card$ meaning cardinality, $E_k \cdot E_j = 1$ whenever $E_k \cap E_j
\neq \emptyset$ and $k \neq j$, and $E_k \cdot E_j = 0$ otherwise, $j,k \in \{1,2, \ldots, n\}$. By Proposition
\ref{CINCO} we see that the inequality we have to prove is
\begin{equation}
\label{uno} - \lfloor \iota \nu (\varphi_{-2}) \rfloor - \lfloor
\iota \nu (\varphi_{-1}) \rfloor- \lfloor \iota \nu (\varphi_1)
\rfloor + 2 \lfloor \iota \nu (\varphi) \rfloor \geq 2,
\end{equation}
$\varphi=\varphi_0$ being an $F_g$-general element for $\nu$ and
$\varphi_l$  an $F'_{g+l}$-general element for $\nu$, $l\in
\{-2,-1,1\}$. In fact, we shall prove that the equality holds in
(\ref{uno}).

Set, as above, $\{\bar{\beta}_i\}_{i=0}^{g+1}$ the sequence of
maximal contact values of $\nu$ and denote by
$\bar{\beta}_i^{\varphi_l}$, $l\in \{-2,-1,1\}$, the $i$-th maximal
contact value of the divisorial valuation $\nu_{\varphi_l}$ defined
by $F'_{g+l}$ (notice that $\varphi_l$ is a general element of this
valuation).
Also, remind that $e_i := \gcd (\bar{\beta}_0, \ldots,
\bar{\beta}_i)$ and $n_i:= e_{i-1}/e_i$, and  denote with a
super-index $\varphi_l$  the analogous values attached to
$\nu_{\varphi_l}$.

By \cite{d} and taking into account that for $h \in R$, $$ \nu(h) =
\min \{(h, \psi) | \psi \mbox{ is a general element of $\nu$} \},$$
$(h, \psi)$ being the intersection multiplicity of the germs given
by $h$ and $\psi$ \cite{spi}, we get the following equalities:
\[
\nu(\varphi_{-2})=e_{g-1} \bar{\beta}_g^{\varphi_{-2}};
\]
\[
\nu(\varphi_{-1})=e_{g-1}^{\varphi_{-1}} \bar{\beta}_g;
\]
\[
\nu(\varphi_{1})=e_{g-1}^{\varphi_{1}} \bar{\beta}_g + 1;
\]
\[
\nu(\varphi)=e_{g-1}^{\varphi} \bar{\beta}_g = n_g \bar{\beta}_g .
\]

Now, we state a result which will be useful in the proof.

\begin{lem}
\label{SEIS} $n_g^{\varphi_{-2}} \bar{\beta}_g - n_g
\bar{\beta}_g^{\varphi_{-2}} =1$.
\end{lem}

\proof Set $\beta^{\prime}_i$ and
${{\beta}^{\prime}_i}^{\varphi_{-2}}$ the Puiseux exponents of the
valuations $\nu$ and $\nu_{\varphi_{-2}}$ (see \cite{spi} for the
definition). Lemma 1.8 in \cite{d-g-n} proves that

\begin{equation}\label{pirri1}
\frac{\bar{\beta}_g}{n_g} = \beta^{\prime}_g + \frac{n_{g-1}}{n_g}
\bar{\beta}_{g-1}-1
\end{equation}
and, analogously,
\[
\frac{\bar{\beta}_g^{\varphi_{-2}}}{n_g^{\varphi_{-2}}} =
 {\beta^{\prime}_g}^{\varphi_{-2}}+ \frac{n_{g-1}^{\varphi_{-2}}}{n_g^{\varphi_{-2}}}
  \bar{\beta}_{g-1}^{\varphi_{-2}}-1.
\]
Since $\bar{\beta}_i^{\varphi_{-2}} = \kappa \bar{\beta}_i$, for $i <g$, where
$\kappa= \bar{\beta}_0^{\varphi_{-2}}/ \bar{\beta}_0$ (see
\cite{cam}, for example), one gets
\begin{equation}\label{pirri2}
\frac{n_{g-1}}{n_g} \bar{\beta}_{g-1} =
\frac{n_{g-1}^{\varphi_{-2}}}{n_g^{\varphi_{-2}}}
\bar{\beta}_{g-1}^{\varphi_{-2}}
\end{equation}
since $e_g = e_g^{\varphi_{-2}}=1$. So
\[
\frac{\bar{\beta}_g}{n_g} -
\frac{\bar{\beta}_g^{\varphi_{-2}}}{n_g^{\varphi_{-2}}} = \beta'_g -
{\beta^{\prime}_g}^{\varphi_{-2}}.
\]
Bearing in mind that $\beta'_g$ and
${\beta^{\prime}_g}^{\varphi_{-2}}$ are consecutive convergents of a
finite continued fraction whose denominators are respectively $n_g$
and $n_g^{\varphi_{-2}}$, by \cite[Th. 7.5]{ni-zu} the equality
\begin{equation}\label{pirri3}
\beta'_g - {\beta^{\prime}_g}^{\varphi_{-2}} = \frac{1}{n_g
n_g^{\varphi_{-2}}}
\end{equation}
holds. The statement follows from (\ref{pirri1}), (\ref{pirri2}) and
(\ref{pirri3}). $\Box$
\\[2mm]

Returning to the proof of our theorem, we recall  that in \cite{ja}
it is proved that $\iota \in \mathcal{H}_g$ has the form
$\iota(g,p,q,r)$ (see page \pageref{FOUR} in this paper).  Since $e_g = 1$ we get
$$
\iota = \frac{(p+rn_g)\bar{\beta}_g+ qn_g}{n_g\bar{\beta}_g}.
$$
For simplicity's sake we set $s:=p+rn_g$. Now stand $\alpha$ and
$\beta$ for $\alpha:= \bar{\beta}_g^{\varphi_{-2}}/ \bar{\beta}_g$
and $\beta:= e_{g-1}^{\varphi_{-1}}/n_g$. Inequality (\ref{uno}) to
be proved can be expressed
\begin{equation}
\label{dos} -\lfloor (qn_g + s \bar{\beta}_g) \alpha \rfloor -
\lfloor (qn_g + s \bar{\beta}_g) \beta \rfloor  -\lfloor (qn_g + s
\bar{\beta}_g) (1+ \frac{1}{n_g \bar{\beta}_g} )\rfloor + 2 (qn_g +
s \bar{\beta}_g) \geq 2.
\end{equation}
For $\iota \in \mathcal{H}_g$, it is necessary that
$(p/n_g)+(q/\bar{\beta}_g) \leq 1$, which is true if and only if
$p\bar{\beta}_g+q n_g \leq n_g \bar{\beta}_g$. However, equality can
not happen in this case because $p \geq 1$, $q\geq 1$ and
$\gcd(\bar{\beta}_g,n_g)=1$.

Let $C$ be a germ of curve given by a general element of $\nu$. It
holds that $\pi^* C= \tilde{C} + D$, where $\tilde{C}$ denotes the
strict transform of $C$ by $\pi$ and $D$ the attached to $\wp$ above
mentioned divisor. Then
\[
(\pi^*C) \cdot F_g = \tilde{C} \cdot F_g + D \cdot F_g =0.
\]
Since $(-D) \cdot F_g = -\nu(\varphi_{-2}) -\nu(\varphi_{-1})
-\nu(\varphi_{1}) + 2 \nu(\varphi)$, we obtain
\begin{equation}
\label{cinco} \alpha + \beta = 1 - \frac{1}{n_g \bar{\beta}_g }.
\end{equation}

Set $\alpha_1 := \alpha + (1/n_g \bar{\beta}_g)$ and $\beta_1 :=
\beta$. By Lemma \ref{SEIS}, one has
\[
(qn_g + s \bar{\beta}_g) \alpha =\frac{q n_g
\bar{\beta}_{g}^{\varphi_{-2}}}{\bar{\beta}_g} + s
\bar{\beta}_{g}^{\varphi_{-2}} = q n_{g}^{\varphi_{-2}} + s
\bar{\beta}_{g}^{\varphi_{-2}} - \frac{q}{\bar{\beta}_g}. \] Thus,
\[
(qn_g + s \bar{\beta}_g) \alpha_1 = q n_{g}^{\varphi_{-2}} + s
\bar{\beta}_{g}^{\varphi_{-2}} - \frac{q}{\bar{\beta}_g} +
\frac{qn_g + s \bar{\beta}_g}{n_g \bar{\beta}_g} = q
n_{g}^{\varphi_{-2}} + s \bar{\beta}_{g}^{\varphi_{-2}} +
\frac{s}{n_g}.
\]
Recall that $s=p + rn_g$, $p< e_{g-1}$, $q< \bar{\beta}_g$ and $p
\bar{\beta}_g + q n_g < n_g \bar{\beta}_g$. Then
\begin{equation}
\label{estrella1} \lfloor (qn_g + s \bar{\beta}_g) \alpha_1 \rfloor
+ \lfloor (qn_g + s \bar{\beta}_g) \beta_1\rfloor = qn_g + s
\bar{\beta}_g -1,
\end{equation}
\begin{equation}
\label{estrella2} \lfloor (qn_g + s \bar{\beta}_g) \alpha_1 \rfloor
= q n_{g}^{\varphi_{-2}} + s \bar{\beta}_{g}^{\varphi_{-2}} +r,
\end{equation}
\begin{equation}
\label{estrella3} \lfloor (qn_g + s \bar{\beta}_g) \alpha  \rfloor
= q n_{g}^{\varphi_{-2}} + s \bar{\beta}_{g}^{\varphi_{-2}} -1
\end{equation}
and
\begin{equation}
\label{estrella4} \lfloor (qn_g + s \bar{\beta}_g) (1+
\frac{1}{n_g\bar{\beta}_g}) \rfloor = \lfloor q n_{g}+ s
\bar{\beta}_{g} + \frac{q n_{g}+ s \bar{\beta}_{g}}{n_g
\bar{\beta}_{g}} \rfloor =
\end{equation}
\[
= \lfloor q n_{g}+ s \bar{\beta}_{g} + \frac{q}{\bar{\beta}_{g}} +
\frac{p}{n_g}+r \rfloor = qn_g + s \bar{\beta}_g +r.
\]
From (\ref{estrella1}), (\ref{estrella2}) and (\ref{estrella3}), we
get
\[
-\lfloor (qn_g + s \bar{\beta}_g) \alpha \rfloor - \lfloor (qn_g + s
\bar{\beta}_g)\beta \rfloor = -(qn_g + s \bar{\beta}_g) +r +2.
\]
This concludes the proof of this case since now (\ref{estrella4})
proves equality in (\ref{dos}).
\begin{itemize}
    \item[{\bf b.}] Now, we are going to prove that for any $i <g$ the divisor $F_i$
     contributes  any $\iota \in {\mathcal H}_i$.
\end{itemize}

With the same conventions and notations as above, consider the
subgraph of $\Gamma$ of vertices corresponding  to $F_i$ and those divisors which meet it (see Figure
\ref{fig2}).
\begin{figure}[h]
$$
\unitlength=1.00mm
\begin{picture}(80.00,30.00)(0,0)
\thicklines \put(30,20){\line(1,0){20}} \put(40,11){\line(0,1){8}}
\put(30,20){\circle*{1}} \put(40,20){\circle*{1}}
\put(50,20){\circle*{1}} \put(40,10){\circle*{1}}
\put(28,22){$F'_{i-2}$} \put(38,22){$F_{i}$} \put(48,22){$F'_{i+1}$}
\put(42,10){$F'_{i-1}$} \put(22,19){$\cdots$} \put(54,19){$\cdots$}
\put(40,4){$\vdots$}
\end{picture}
$$
\caption{$F_i$ contributes  $
{\mathcal H}_i$}
\label{fig2}
\end{figure}
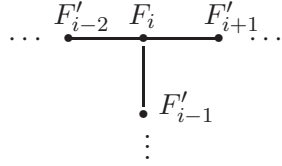

\begin{itemize}
    \item[{\bf b1.}] Firstly, we assume that the divisor $F'_{i+1}$ is free.
\end{itemize}
Set $\vv$ the valuation given by the last free divisor represented
in $\Gamma_{i+1}$. Notice that for any analytically irreducible
element $h$ in $R$ whose strict transform (the associated germ of
curve) cuts  transversally some of the divisors created to define
$\vv$, the equality $\nu(h) = (\bar{\beta}_0^\nu /
\bar{\beta}_0^\vv) \vv(h)$ holds. Along this proof, we set $\kappa:=
\bar{\beta}_0^\nu / \bar{\beta}_0^\vv$ and any jumping number
$\iota^\nu=\iota(i,p,q,r)$ in the set $\mathcal{H}_i$ satisfies:
\[
\iota^\nu = \frac{(p+rn_i^\nu)\bar{\beta}_i^\nu+ q
e^\nu_{i-1}}{e^\nu_{i-1}\bar{\beta}_i^\nu} = \frac{(p+rn_i^\vv)
\kappa \bar{\beta}_i^\vv+ q \kappa e^\vv_{i-1}}{\kappa^2
e^\vv_{i-1}\bar{\beta}_i^\vv} = \frac{1}{\kappa} \iota^\vv,
\]
where the super-indices $\nu$ and $\vv$ make reference to the
valuation that corresponds to the used values, and $\iota^\vv$ is
the jumping number
$[(p+rn_i^\vv)\bar{\beta}_i^\vv+qe_{i-1}^\vv]/(e_{i-1}^\vv
\bar{\beta}_i^{\vv})$ of the simple complete ideal defined by $\vv$,
which belongs to the corresponding set ${\mathcal H}_i^{\vv}$.

Now, our result is proved  because inequality (\ref{uno}) in our
case coincides with the same inequality for $\vv$, which
accomplishes it since the valuation $\vv$ is in the situation 1a.

\begin{itemize}
    \item[{\bf b2.}] Now suppose that $F'_{i+1}$ is a satellite divisor.
\end{itemize}
We can assume that $i=g-1$, because when $i<g-1$ a reasoning as
in 1b1 would finish the proof.

Again, set $\vv$ the valuation defined by the free divisor
represented in the graph $\Gamma_g$. Keeping our notation, we must
prove
\[
 - \lfloor \iota \nu (\varphi_{-2}) \rfloor - \lfloor \iota \nu
(\varphi_{-1}) \rfloor- \lfloor \iota \nu (\varphi_1) \rfloor +
(t+1)  \iota \nu (\varphi) \geq 2,
\]
where $t$ is the cardinality of the set of  prime
exceptional divisors in (\ref{spip}) proximate to $F_{g-1}$. The worst case
happens when $t=2$, so we assume it. By using the valuation $\vv$
and taking into account that $\iota^\nu = \frac{1}{\kappa}
\iota^\vv$, the reasoning in 1a shows that
\[
 \iota \nu (\varphi)   - \lfloor \iota \nu
(\varphi_{-2}) \rfloor -  \lfloor \iota \nu (\varphi_{-1}) \rfloor
\geq 2 + r.
\]
Looking at the sequence of values for $\nu$ (see
\cite[1.5.1]{d-g-n}), set $a$ ($b \geq a$, respectively) for the
value at the divisor where the strict transform of the germ given by
$\varphi_1$ intersects transversally (for the value at the defining
divisor of $\vv$, respectively). Then
\[
  \iota \nu (\varphi)   = q e_{g-2}^\vv+ s
\bar{\beta}_{g-1}^\vv,
\]
where $s$ is obtained as above. Now, since \begin{equation}
\label{cuatro} \nu (\varphi_1) = 2 \nu (\varphi) + a + b,
\end{equation}
because $t=2$, we get
\begin{equation}
\label{cuatroprima} \iota \nu (\varphi_1) = 2 (q e_{g-2}^\vv+ s
\bar{\beta}_{g-1}^\vv) + \iota a + \iota b = 2 (q
e_{g-2}^\vv+ s \bar{\beta}_{g-1}^\vv) \left[1 + \frac{a+b}{2 \kappa
e_{g-2}^\vv \bar{\beta}_{g-1}^\vv}\right].
\end{equation}
It is clear that $(a+b) \bar{\beta}_{0}^\vv = \bar{\beta}_{0}^\nu$
and then the right hand side in (\ref{cuatroprima}) equals
\[
 2 (q e_{g-2}^\vv+ s
\bar{\beta}_{g-1}^\vv) \left[1 + \frac{a+b}{2(a+b) e_{g-2}^\vv
\bar{\beta}_{g-1}^\vv}\right] = 2 (q e_{g-2}^\vv+ s
\bar{\beta}_{g-1}^\vv) \left[1 + \frac{1}{2 e_{g-2}^\vv
\bar{\beta}_{g-1}^\vv}\right].
\]
So, $\lfloor \iota \nu(\varphi_1) \rfloor \leq 2 (q e_{g-2}^\vv+ s
\bar{\beta}_{g-1}^\vv) + r $ since
\[
\frac{(q e_{g-2}^\vv+ s \bar{\beta}_{g-1}^\vv)}{e_{g-2}^\vv
\bar{\beta}_{g-1}^\vv}
\]
is the jumping number $\iota^\vv$ that, by the conditions given in
\cite{ja}, is less than or equal to $r$.

It only remains to prove that
\begin{itemize}
    \item[{\bf c.}] $F_{g+1}$ contributes any jumping number $\iota \in
    \mathcal{H}_{g+1}$.
\end{itemize}
Indeed, $\iota$ has the form $\iota= (p/e_g)
+(q/\bar{\beta}_{g+1})$, $p,q
> 1$ and $e_g=1$, so $\iota= (q+ p \bar{\beta}_{g+1})/
\bar{\beta}_{g+1}$. Keeping the above notations, the subgraph we are
interested in is the one depicted in Figure \ref{fig3}.
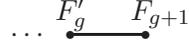
\begin{figure}[h]
$$
\unitlength=1.00mm
\begin{picture}(80.00,20.00)(0,0)
\thicklines \put(30,10){\line(1,0){10}}
\put(30,10){\circle*{1}} \put(40,10){\circle*{1}}
\put(28,12){$F'_{g}$} \put(38,12){$F_{g+1}$}
\put(22,9){$\cdots$} 
\end{picture}
$$
\caption{$F_{g+1}$ contributes  $\mathcal{H}_{g+1}$} \label{fig3}
\end{figure}
As above set  $\varphi$ and $\varphi_{-1}$ analytically
irreducible elements of type $\varphi_i$ attached to the divisors
$F_{g+1}$ and $F'_{g}$ respectively. Then, we must prove
\[
- \lfloor \iota \nu(\varphi_{-1}) \rfloor +  \iota \nu(\varphi)
 \geq 2.
\]
And this happens since
\[
\iota \nu(\varphi_{-1}) = \frac{q+ p
\bar{\beta}_{g+1}}{\bar{\beta}_{g+1}} (\bar{\beta}_{g+1} -1),
\]
$ \iota \nu(\varphi) = q+ p \bar{\beta}_{g+1}$ and $p \geq 1$.

\begin{itemize}
    \item[{\bf 2.}] To end the proof of the direct implication, let us assume that  $g^* =g-1$.
\end{itemize}
We shall only prove that the divisor $F_g$ contributes any $\iota
\in \mathcal{H}_g$. A proof for the case of the remaining divisors
$F_i$, $i<g$, works similarly to the analogous case in 1.

Here, the subgraph of $\Gamma$ of vertices corresponding to $F_g$ and divisors meeting it will be
the one in Figure \ref{fig4}
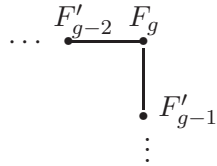
\begin{figure}[h]
$$
\unitlength=1.00mm
\begin{picture}(80.00,30.00)(0,0)
\thicklines \put(30,20){\line(1,0){10}} \put(40,11){\line(0,1){8}}
\put(30,20){\circle*{1}} \put(40,20){\circle*{1}}
\put(40,10){\circle*{1}} \put(28,22){$F'_{g-2}$}
\put(38,22){$F_{g}$}
\put(42,10){$F'_{g-1}$} \put(22,19){$\cdots$}
\put(40,4){$\vdots$}
\end{picture}
$$
\caption{$F_g$ contributes  $\mathcal{H}_g$, when $g^* = g-1$}
\label{fig4}
\end{figure}
and with the same notations of 1a, we must prove
\[
- \lfloor \iota \nu (\varphi_{-2}) \rfloor - \lfloor \iota \nu
(\varphi_{-1}) \rfloor + \iota \nu (\varphi)  \geq 2,
\]
that is
\[
- \lfloor(q n_{g}+ s \bar{\beta}_{g}) \alpha \rfloor  - \lfloor(q
n_{g}+ s \bar{\beta}_{g}) \beta \rfloor +q n_{g}+ s \bar{\beta}_{g}
\geq 2.
\]
Finally, since
\[
1= (-D) \cdot F_g = - \nu (\varphi_{-2}) - \nu (\varphi_{-1}) + \nu
(\varphi),
\]
the same reasoning we did to show (\ref{cinco}) allows to set $n_{g}
\bar{\beta}_{g} (-\alpha -\beta +1)=1$, where $\alpha$ and $\beta$
are as above, and so we also get $\alpha + \beta = 1 - (1/n_{g}
\bar{\beta}_{g})$, i.e.,
$n_g\bar{\beta}_{g}^{\varphi_{-2}}+\bar{\beta}_{g}e_{g-1}^{\varphi_{-1}}=n_g\bar{\beta}_{g}-1$.
Then,
$$ \lfloor(q n_{g}+ s \bar{\beta}_{g}) \alpha \rfloor + \lfloor(q
n_{g}+ s \bar{\beta}_{g}) \beta\rfloor -(qn_g+s\bar{\beta}_{g})=$$

$$=\lfloor \frac{q\bar{\beta}_{g}^{\varphi_{-2}}n_g}{\bar{\beta}_{g}}+s\bar{\beta}_{g}^{\varphi_{-2}} \rfloor + \lfloor qe_{g-1}^{\varphi_{-1}}+ \frac{se_{g-1}^{\varphi_{-1}}\bar{\beta}_{g}}{n_g} \rfloor -(qn_g+s\bar{\beta}_{g})=$$

$$ =s\bar{\beta}_{g}^{\varphi_{-2}}+qe_{g-1}^{\varphi_{-1}}+ \lfloor \frac{q(n_g\bar{\beta}_{g}-1)}{\bar{\beta}_{g}}-qe_{g-1}^{\varphi_{-1}}\rfloor + \lfloor \frac{s(n_g\bar{\beta}_{g}-1)}{n_g}-s\bar{\beta}_{g}^{\varphi_{-2}}\rfloor - (qn_g+s\bar{\beta}_{g}) =$$

$$= \lfloor qn_g-\frac{q}{\bar{\beta}_{g}}\rfloor + \lfloor s\bar{\beta}_{g}-\frac{s}{n_g}
\rfloor - (qn_g+s\bar{\beta}_{g}) \leq (qn_g-1) + (s\bar{\beta}_{g}-1)-(qn_g+s\bar{\beta}_{g})
=-2.$$
\\

In order to prove the converse implication we shall consider two
previous lemmas.

\begin{lem}\label{ideafix}
Consider an $n$-tuple of non-negative integers $(\alpha_1,\alpha_2,\ldots,
\alpha_n)$ and let $D(\alpha)$ be the divisor $-\sum_{j=1}^n
\alpha_j E_j$. For every nonempty subset $\{F_{i_1}, F_{i_2},\ldots,
F_{i_t}\}$ of the set of divisors $\{F_i\}_{i=1}^{g^*+1}$, the following equality holds:
\begin{equation}\label{DIM}
\dim_{\mathbb{C}} \frac{\pi_* \mathcal{O}_X (D(\alpha)+ \sum_{l=1}^t
F_{i_l})}{\pi_* \mathcal{O}_X (D(\alpha))} = \sum_{l=1}^t
\dim_{\mathbb{C}} \frac{\pi_* \mathcal{O}_X
(D(\alpha)+F_{i_l})}{\pi_* \mathcal{O}_X (D(\alpha))}.
\end{equation}
\end{lem}

\begin{proof}
Without loss of generality  we can assume that $i_1<i_2<\cdots<i_t$. In a first step, we prove
(\ref{DIM}) for $t=2$.

Consider the following commutative diagram of ideals in $R$ and
injective maps, where, for any sum $G$ of divisors $E_j$, we stand
$\E_\alpha (G)$ for $\pi_* \mathcal{O}_X (D(\alpha)+ G)$ and an
expression like $A \xrightarrow[]{{[p]}} B$ means that the dimension
of the vector space quotient $B/A$ equals $p$.

\[
\begin{CD}
\E_\alpha :=\E_\alpha (\emptyset)  @> [p_{i_1}]>> \E_\alpha
(F_{i_1}) \\
@V[p_{i_2}]VV @V [q_{i_2}]VV \\
\E_\alpha (F_{i_2}) @> [q_{i_1}]>> \E_\alpha (F_{i_1}+ F_{i_2}).
\end{CD}
\]
\\

We only need to prove that
\begin{equation}
\label{EQ}
 p_{i_1} = q_{i_1},
\end{equation}

because it holds the following vector space isomorphism
\[
\frac{ \E_\alpha(F_{i_1} + F_{i_2})}{ \E_\alpha } \cong
\frac{\E_\alpha ( F_{i_1} + F_{i_2})}{\E_\alpha ( F_{i_2})} \oplus
\frac {\E_\alpha ( F_{i_2})}{\E_\alpha}.
\]
The symmetric isomorphism given by the diagram proves that $ p_{i_2}
= q_{i_2}$ is also true.

(\ref{EQ}) holds when either $\alpha_{st_{i_1}}=0$ or
$\alpha_{st_{i_2}}=0$, therefore we can assume that both values are
positive. Set $\nu_{i_1}$ and $\nu_{i_2}$ the divisorial valuations
defined by $F_{i_1}$ and $F_{i_2}$, respectively. By the proof of
\cite[Th. 1]{ga}, a basis of the vector space $\E_\alpha(F_{i_1})/
\E_\alpha $ is given by classes defined by ``monomials" of the type
$\prod_{k=0}^{g'+1} \varphi_{l_k}^{a_k}$, $a_k\geq 0$, where $l_k$
and $\varphi_{l_k}$, $0\leq k\leq g'+1$, are elements as in the
paragraph after Definition \ref{general} but associated with the
valuation $\nu_{i_1}$ (in particular $\varphi_{l_{g'+1}}$ denotes a
general element of $\nu_{i_1}$). Clearly
$\nu_{i_1}(\prod_{k=0}^{g'+1} \varphi_{l_k}^{a_k}) =
\alpha_{st_{i_1}} -1$. Taking into account that $F_{i_2}$
corresponds either to a star vertex (as $F_{i_1}$) or to the last
vertex of  the dual graph of $\nu$ (always denoted by $st_{i_2}$)
and that $i_1 < i_2$, it holds that all generator ``monomials"
$\prod_{k=0}^{g'+1} \varphi_{l_k}^{a_k}$ as above have the same
valuation $\nu_{i_2}$, that is a multiple of $\alpha_{st_{i_1}} -1$
and larger than or equal to $\alpha_{st_{i_2}}$. This shows that
$p_{i_1} = q_{i_1}$ because the classes in $\E_\alpha(F_{i_1} +
F_{i_2})/ \E_\alpha$ of the ``monomials" spanning
$\E_\alpha(F_{i_1})/ \E_\alpha$ are linearly independent elements
and one cannot find any element $\prod_{k=0}^{g'+1}
\varphi_{l_k}^{a_k}$ such that $\nu_{i_1}(\prod_{k=0}^{g'+1}
\varphi_{l_k}^{a_k})=\alpha_{st_{i_1}} -1$ and
$\nu_{i_2}(\prod_{k=0}^{g'+1} \varphi_{l_k}^{a_k})=\alpha_{st_{i_2}}
-1$.

Notice that in  the above reasoning, the specific values
$\alpha_j$ for those indices not corresponding to the divisors
defining the valuations $\nu_{i_1}$ and $\nu_{i_2}$ are not
relevant.

When $t>2$, we can reduce the proof to the above situation. Indeed,
if we consider the diagram in Figure \ref{conmutativo},
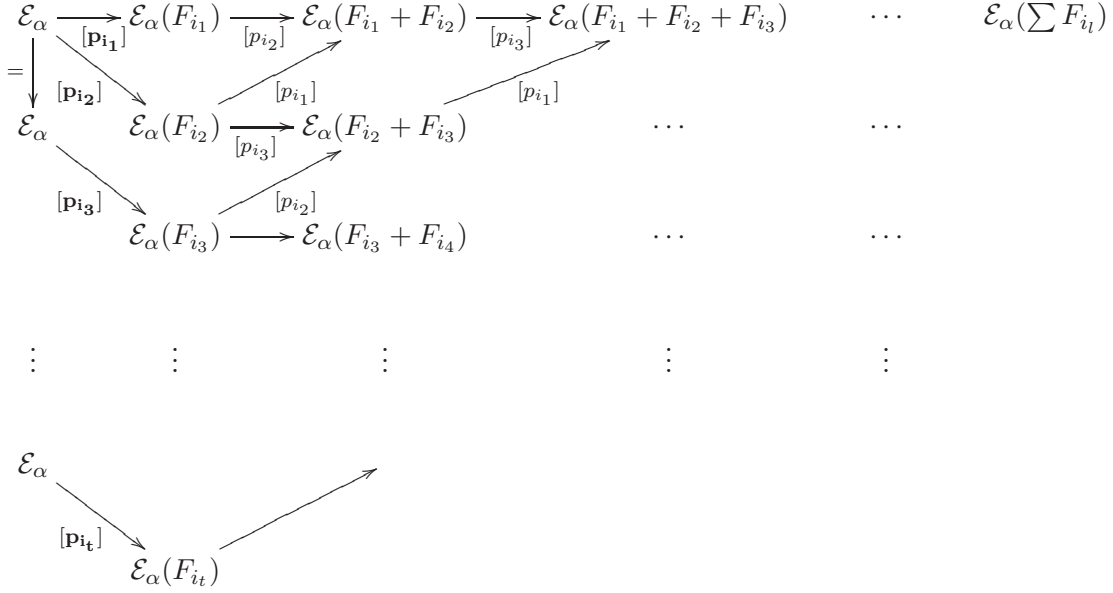
\begin{figure}
\begin{center}
  \xymatrix{ & & & & &\\ \E_\alpha \ar[d]_{=}
\ar[r]_{\bf{[p_{i_1}]}} \ar[dr]_{\bf{[p_{i_2}]}}
&\E_\alpha(F_{i_1})\ar[r]_<<<<<{[p_{i_2}]} & \E_\alpha(F_{i_1} +
F_{i_2}) \ar[r]_<<<<<{[p_{i_3}]} &
\E_\alpha(F_{i_1}+ F_{i_2}+F_{i_3})& \cdots &\E_\alpha(\sum F_{i_l})\\
\E_\alpha  \ar[dr]_{\bf{[p_{i_3}]}} & \E_\alpha(F_{i_2})
\ar[r]_<<<<{[p_{i_3}]} \ar[ur]_{[p_{i_1}]} & \E_\alpha(F_{i_2}+
F_{i_3})
 \ar[ur]_{[p_{i_1}]} & \cdots  & \cdots  & \\
 & \E_\alpha(F_{i_3}) \ar[r] \ar[ur]_{[p_{i_2}]}& \E_\alpha(F_{i_3}+ F_{i_4}) & \cdots & \cdots & \\
\vdots & \vdots & \vdots& \vdots & \vdots &\\
\E_\alpha \ar[dr]_{\bf{[p_{i_t}]}}&&&&& \\ &\E_\alpha(F_{i_t})
\ar[ur]& & & &\\}
\end{center}
\caption{Diagram for $t>2$} \label{conmutativo}
\end{figure}
we can do the above reasoning for the sub-diagrams including three
consecutive ideals of the type described for $t=2$ (changing
$D(\alpha)$ by  sums of the type $D(\alpha) + \sum_{l\in L} F_l$,
with a suitable set of indices $L$). We have set in boldface letter
the dimensions we know and without this type of letter those that we
compute in an iterative manner by using the reasoning for $t=2$.
This implies that we can compute the dimensions appearing in the
arrows of the top line of the diagram. The sum of those dimensions
is the dimension we desire to compute and it is $p_{i_1} + p_{i_2} +
\cdots + p_{i_t}$.
\end{proof}

We shall use the above lemma to prove the following one, which shows
that Theorem \ref{DOS} holds for less than 1 jumping numbers. Before
stating it, we recall that the less than 1 jumping numbers of $\wp$
(and their multiplier ideals) coincide with those associated with
the curve in ${\rm Spec}(R)$ defined by  a general element of the
valuation $\nu$, $\varphi$, \cite[Prop. 9.3]{ja}. Furthermore,
Varchenko \cite{var} (see also \cite{bu}) proved that these jumping
numbers are exactly the exponents $\alpha$ in the interval $]0,1[$
corresponding to non-zero terms of the Hodge spectrum of $\varphi$,
${\rm Sp}(\varphi)=\sum_{\alpha\in \mathbb{Q}}
n_{\alpha}(\varphi)t^{\alpha}$, which is a fractional Laurent
polynomial with integer coefficients that can be defined using the
mixed Hodge structure and the monodromy on the cohomology of the
Milnor fiber of $\varphi$. We must mention here that, although the
Hodge spectrum of a hypersurface singularity was first defined in
\cite{ste1} and \cite{ste2}, we consider the definition used in
\cite{sai2,sai,bu}.

\begin{lem}\label{asterix}
If $\iota$ is a jumping number of a simple complete ideal $\wp$ such that $0<\iota<1$
and a divisor $F_i$ ($1\leq i\leq g^*+1$) contributes $\iota$, then
$\iota \in {\mathcal H}_i$.
\end{lem}
\begin{proof}

Assume that $F_i$ contributes $\iota$. Consider the set of indices
$\Pi:=\{k \mid 1\leq k \leq g^*+1 \mbox{ and } \iota \in {\mathcal
H}_k\}$. Taking into account that $F_k$ contributes $\iota$ for all
$k\in \Pi$ (by the direct implication of Theorem \ref{DOS}) one has
$$
{\mathcal J}
\left(\wp^{\iota} \right) \subseteq \pi_*\mathcal{O}_X \left( - \lfloor \iota D
\rfloor + K_{X|X_0} + \sum_{k \in \Pi} F_{k} \right) \subseteq {\mathcal J}
\left(\wp^{\iota^{<}} \right).
$$

By \cite{bu}, it happens that
\begin{equation}\label{ppp}
n_{\iota}(\varphi)=\dim_{\mathbb{C}} {\mathcal J}
(\wp^{\iota^{<}}) / {\mathcal J}(\wp^{\iota}),
\end{equation}
and, by \cite[1.5]{sai}, $n_{\iota}(\varphi)$ is the cardinality of
$\Pi$. Therefore, by Lemma \ref{ideafix}, the equality
$\pi_*\mathcal{O}_X ( - \lfloor \iota D \rfloor + K_{X|X_0} +
\sum_{k\in \Pi} F_{k})= {\mathcal J} (\wp^{\iota^{<}})$ holds. As a
consequence if $i$ would not be in $\Pi$, then ${\mathcal J}
(\wp^{\iota^{<}})=\pi_*\mathcal{O}_X ( - \lfloor \iota D \rfloor +
K_{X|X_0} + \sum_{k\in \Pi} F_{k}+F_i)$, which is a contradiction
with Lemma \ref{ideafix} and (\ref{ppp}).
\end{proof}

\noindent {\it Remark.} Notice that, from the above proof, it
follows that the dimension of the quotient $\pi_*\mathcal{O}_X ( -
\lfloor \iota D \rfloor + K_{X|X_0} + F_i)/ {\mathcal
J}(\wp^{\iota})$ is 1 whenever $\iota \in {\mathcal H}_i$, $1\leq i
\leq g^*+1$, and
 $0<\iota<1$. \\

Now we shall take advantage of Lemma \ref{asterix} to prove that its
statement is true for whichever jumping number $\iota$ of $\wp$.
Notice that $\iota\not=1$ by \cite[Prop. 8.9]{ja}. Consider a
divisor $F_i$, $1\leq i\leq g^*+1$, such that $F_i$ contributes
$\iota>1$ and let us prove that $\iota \in {\mathcal
H}_i$.\\

If $\iota$ is an integer then $\iota\in {\mathcal H}_{g^*+1}$ (by
\cite[Prop. 8.11]{ja}). As $\iota D \cdot F_i = \lfloor \iota D
\rfloor \cdot F_i\leq -2$ (by Proposition \ref{CINCO}), $i$ must be
$g^*+1$ since otherwise the above equality would not happen because
$D\cdot E_j=0$ whenever $E_j\not=F_{g^*+1}$.\\

If $i=g^*+1$ then $\iota \nu(\wp)$ is a positive integer and,
therefore, $\iota
\in {\mathcal H}_{g^*+1}$ (again by \cite[Prop. 8.11]{ja}).\\

Hence we can assume from now on that $\iota$ is not an integer and
$i\not=g^*+1$. Set $a:=\lfloor \iota \rfloor$ and
$\beta:=\iota-a$. We shall distinguish two cases, proving that the second one cannot hold.\\

\noindent {\it Case 1.} $\beta$ is a jumping number of $\wp$. Obviously $\beta$ is a candidate jumping number
from $F_i$. We have $\lfloor \beta D \rfloor \cdot F_i= \lfloor
\iota D \rfloor \cdot F_i- a D\cdot F_i=\lfloor \iota D \rfloor
\cdot F_i\leq -2$ (by Proposition \ref{CINCO}) and, thus,
$F_i$ contributes $\beta$. Since $0<\beta<1$ we apply Lemma
\ref{asterix} concluding that $\iota=a+\beta\in {\mathcal H}_i$.\\

\noindent {\it Case 2.} $\beta$ is not a jumping number. $F_i$
contributes $\iota$ so $-\lfloor \beta D \rfloor \cdot F_i=-\lfloor
\iota D \rfloor \cdot F_i + aD\cdot F_i=-\lfloor \iota D \rfloor
\cdot F_i \geq 2$. Now, $\pi_*\mathcal{O}_X ( - \lfloor \beta D
\rfloor + K_{X|X_0} + F_i)\not=\pi_*\mathcal{O}_X ( - \lfloor \beta
D \rfloor + K_{X|X_0})$ by Proposition \ref{CINCO} (a). Since $\beta
\nu(\varphi_i)$ is an integer number it holds that
$\pi_*\mathcal{O}_X ( - \lfloor (\beta-\epsilon) D \rfloor +
K_{X|X_0})\not= \pi_*\mathcal{O}_X ( - \lfloor \beta D \rfloor +
K_{X|X_0})$ for all $\epsilon>0$. Then $\beta$ should be
a jumping number, which is a contradiction.\\

This concludes the proof of Theorem \ref{DOS}. $\Box$

\subsection{Proof of Theorem \ref{TRES}}
Next we prove Theorem \ref{TRES}, which states that if $\iota$ is
a jumping number of a simple complete ideal $\wp$, then\[
\pi_*\mathcal{O}_X \left( - \lfloor \iota D \rfloor + K_{X|X_0} +
\sum_{l=1}^s F_{i_l} \right) = \mathcal{J}
\left(\wp^{\iota^<}\right),
\]
where $\{i_1,i_2,\ldots,i_s\}$ is the set of indexes $i$, $1 \leq i \leq g^* +1$, such that
$\iota\in \mathcal{H}_{i}$.

To begin with,  we shall prove some technical lemmas which keep the above notation. The first one
is related to Enriques' ``principle of discharge'' \cite[pag. 28]{z-3}:
\begin{lem}\label{descarga}
Let $F$ be a divisor of $X$ with exceptional support and let $E_k$
be a prime exceptional divisor such that $F\cdot E_k>0$. Then
$\pi_*\mathcal{O}_X(-F)=\pi_*\mathcal{O}_X(-F-E_k).$
\end{lem}
\begin{proof}
It suffices to take global sections on the natural
exact sequence
$$0\rightarrow \mathcal{O}_X(-F-E_k)\rightarrow \mathcal{O}_X(-F)\rightarrow \mathcal{O}_X(-F)\otimes \mathcal{O}_{E_k}.$$
\end{proof}

Our next lemma follows easily from the fact that the equality
$D\cdot E_{k}=0$ happens whenever $k\not=n$.

\begin{lem}\label{a}
Let $E_k$ be a prime exceptional divisor such that $\iota$ is a
candidate jumping number from $E_k$. If $E_k\not=F_i$ for all $i\in
\{1,2,\ldots,g^*+1\}$, then either $\iota$ is a candidate jumping
number from all prime exceptional divisors meeting $E_k$ or $\iota$
is not a candidate jumping number from none of them. $\Box$

\end{lem}

\begin{lem}\label{c}
Let $i\in \{1,2,\ldots,g^*\}$ and let $F'_{i-2}, F'_{i-1}$ and
$F'_{i+1}$ be the three prime exceptional divisors meeting $F_i$,
whose corresponding vertices in $\Gamma$ are depicted in Figure
\ref{fig2}. If $\iota$ is a candidate jumping number from $F_i$ and
$F'_{i+1}$, then $\iota$ is also so from $F'_{i-2}$ and
$F'_{i-1}$.

\end{lem}
\begin{proof}
Notice that, due to the equality $D\cdot F_i=0$ and the fact that
$\iota$ is a candidate jumping number from $F_i$ and $F'_{i+1}$, it
is enough to prove that $\iota$ is a candidate jumping number from
$F'_{i-1}$. With the same notation of the proof of Theorem
\ref{DOS}, we notice that $\nu(\varphi_1)= t \nu(\varphi)+ e_{i}$,
where $t$ is the number of proximate to $F_{i}$ prime exceptional
divisors. Therefore $\iota e_i$ is a positive integer because
$\iota$ is a candidate jumping number from $F_i$ and  $F'_{i+1}$.
Since $\nu(\varphi_{-1})=e_{i-1}^{\varphi_{-1}} \bar{\beta}_i$ and
$e_i$ divides $\bar{\beta}_i$ one has that $\iota \nu(\varphi_{-1})$
is also a positive integer and, hence, $\iota$ is a candidate
jumping number from $F'_{i-1}$.
\end{proof}

From now on, denote by $\Delta$ the set of prime exceptional
divisors (provided by the sequence $\pi$ of (\ref{spip})) from which $\iota$ is a candidate jumping number.

\begin{lem}\label{b}
Let $E_k$ be a divisor in $\Delta$ such that $E_k\not= F_i$ for all
$i\in \{1,2,\ldots,g^*+1\}$ and consider a subset $S$ of $\Delta$
such that $E_k\in S$ and the cardinality of the set ${\mathcal S}_k := \{E_j\in S\mid
E_j\cap E_k\not=\emptyset\}$ is less than or equal to 1.  Then
$$\pi_* {\mathcal O}_X \left(-\lfloor \iota D \rfloor +
K_{X|X_0}+\sum_{E_j\in S} E_j \right)=\pi_* {\mathcal O}_X
\left(-\lfloor \iota D \rfloor + K_{X|X_0}+\sum_{E_j\in S\setminus
\{E_k\}} E_j \right).$$

\end{lem}

\begin{proof}
Set $G$ the divisor $\lfloor \iota D \rfloor -
K_{X|X_0}-\sum_{E_j\in S} E_j$. By Lemma \ref{descarga} it
suffices to prove that $G\cdot E_k>0$. It is clear that $G\cdot
E_k= \lfloor \iota D \rfloor\cdot E_k-K_{X|X_0}\cdot
E_k-E_k^2-\epsilon$, where $\epsilon=0$ ($\epsilon=1$,
respectively) whenever the cardinality of the set ${\mathcal S}_k
$ equals 0 (1, respectively). Taking into account that
$K_{X|X_0}\cdot E_k=-E_k^2-2$, we get
$$G\cdot E_k=\lfloor \iota D \rfloor\cdot E_k+2-\epsilon.$$ When
${\mathcal S}_k $ is empty, the condition
$G\cdot E_k>0$ is equivalent to
$\lfloor \iota D \rfloor\cdot E_k\geq -1$, which is true by
Proposition \ref{CINCO}. Otherwise we must prove that
$\lfloor \iota D \rfloor\cdot E_k\geq 0$. Indeed, by Lemma \ref{a},
$\iota$ is a candidate jumping number from all prime exceptional divisors meeting
$E_k$ and, then, $\lfloor \iota D \rfloor\cdot E_k=\iota D\cdot
E_k=0$.
\end{proof}

Given two prime exceptional divisors $E_k$ and $E_j$, we shall denote by $[E_k,E_j]$
($[E_k,E_j[$, respectively) ($]E_k,E_j[$, respectively) the set of
prime exceptional divisors corresponding to the vertices of the
dual graph $\Gamma$ of $\wp$ which are on the shortest path that
joins the vertices associated with $E_k$ and $E_j$
(the set $[E_k,E_j]\setminus \{E_j\}$, respectively) (the set $]E_k,E_j[\setminus
\{E_k,E_j\}$, respectively). The {\it length} of $[E_k,E_j]$ will
be the length of the mentioned path (that is, its number of
edges).

\begin{lem}\label{d}
Let $S$ be a subset of $\Delta$ and  $E_k$  a  prime exceptional
divisor given by $\pi$ that is different from $ F_i$ for all $i\in
\{1,2,\ldots,g^*+1\}$. Assume that either $E_k$ does not belong to
$S$, or $E_k$ is in $S$ and either the associated with $E_k$ vertex
in the dual graph $\Gamma$ of $\pi$ is a dead one or it is adjacent
to a vertex whose associated divisor does not belong to $S$.
Consider any divisor $F_r$  such that there is no divisor $F_i$
satisfying $F_i\in ]E_k,F_r[$. Then
$$\pi_* {\mathcal O}_X \left(-\lfloor \iota D \rfloor +
K_{X|X_0}+\sum_{E_j\in S} E_j \right)=\pi_* {\mathcal O}_X
\left(-\lfloor \iota D \rfloor + K_{X|X_0}+\sum_{E_j\in S\setminus
[E_k,F_r[} E_j \right).$$
\end{lem}
\begin{proof}
It follows by applying Lemma \ref{b} and making induction on the
length of $[E_k,F_r]$.

\end{proof}

We conclude this subsection  proving Theorem \ref{TRES} with the help of the above
lemmas. It is clear that
$$\mathcal{J} \left(\wp^{\iota^<}\right)=\pi_*\mathcal{O}_X \left( - \lfloor \iota D \rfloor + K_{X|X_0} +
\sum_{E_i\in \Delta} E_i \right).$$ For each $i=1,2,\ldots,g^*+1$,
define $\Delta'_i$ as the set of divisors in $\Delta$ associated
with some vertex of the the subgraph $\Gamma_i$ of  $\Gamma$ such
that they do not contribute  $\iota$. Then, by Theorem \ref{DOS},
\begin{equation}\label{deltaprima}
\Delta\setminus \bigcup_{i=1}^{g^*+1} \Delta'_i=\{F_{i_1},
F_{i_2},\ldots,F_{i_s}\}.
\end{equation}

To prove Theorem \ref{TRES}, it will be enough to show that
\begin{equation}\label{www}
\mathcal{J} \left(\wp^{\iota^<}\right) = \pi_*\mathcal{O}_X \left(
- \lfloor \iota D \rfloor + K_{X|X_0} + \sum_{E_j\in
\Delta\setminus S_i} E_j \right)
\end{equation}
for any $i\in \{1,2,\ldots, g^*+1\}$, where $S_i:=\bigcup_{k=1}^{i}
\Delta'_k$. We shall assume that $g^*>0$ and we shall apply induction on $i$ (the result for $g^*=0$ can
be easily proved using reasonings of the forthcoming induction
procedure).

Let us show (\ref{www}) for $i=1$. Set $E_1$ and $E_t$  the
divisors corresponding to the two dead vertices of the subgraph
$\Gamma_1$. Applying Lemma \ref{d} with $E_k=E_1$ and $S=\Delta$,
it happens
\[
\mathcal{J} \left(\wp^{\iota^<}\right) = \pi_*\mathcal{O}_X \left(
- \lfloor \iota D \rfloor + K_{X|X_0} + \sum_{E_j\in
\Delta\setminus [E_1,F_1[} E_j \right).
\]
Again by Lemma \ref{d} but taking $E_k=E_t$ and $S=\Delta\setminus
[E_1,F_1[$, we obtain
\[
\mathcal{J} \left(\wp^{\iota^<}\right) = \pi_*\mathcal{O}_X \left(
- \lfloor \iota D \rfloor + K_{X|X_0} + \sum_{E_j\in
\Delta\setminus (S_{1}\setminus \{F_{1}\})} E_j \right).
\]

If either $F_1\not\in \Delta$ or $F_1$ contributes $\iota$, then
(\ref{www}) holds for $i=1$ because $S_1\subseteq [E_1,F_1[\cup
[E_t,F_1[$.

If $F_1\in \Delta_1'$ and $F_2'\in \Delta$, using Lemma \ref{c} we
get
$$\left(
\lfloor \iota D \rfloor - K_{X|X_0} - \sum_{E_j\in \Delta\setminus
(S_{1}\setminus \{F_{1}\})} E_j \right)\cdot F_1=\iota D \cdot F_1 -
K_{X|X_0} \cdot F_1- F_1^2-1>0$$ taking into account that
$K_{X|X_0}\cdot F_1=-F_1^2-2$ and  $D\cdot F_1=0$. If, otherwise,
$F_1\in \Delta_1'$ and $F_2'\not\in \Delta$, the fact that the left
hand side of the above equality is also positive follows easily from
Proposition \ref{CINCO}. Thus,
applying Lemma \ref{descarga} we conclude the proof of Equality (\ref{www}) for $i=1$.\\

Assume now that (\ref{www}) is true for $i$, $1\leq i\leq g^*$, and
let us show it for $i+1$. With the notation of Lemma \ref{c}, either
whether
 $\iota$ is a candidate jumping number from the divisors
$F_{i}$ and $F'_{i+1}$ (in which case Lemma \ref{c} shows that
$\lfloor \iota D \rfloor \cdot F_{i-1}= \iota D\cdot F_{i}=0$ what
implies that $F_{i}$ does not contribute $\iota$ (by Proposition
\ref{CINCO}) and, therefore, $F_{i}\not\in \Delta\setminus S_{i}$),
or whether $\iota$ is not a candidate jumping number either from
$F'_{i+1}$ or $F_{i}$, one can apply the induction hypothesis and
Lemma \ref{d} (with $E_k=F'_{i+1}$ and $S=\Delta\setminus  S_i$)
getting the equality
\begin{equation}\label{bbb}
\mathcal{J} \left(\wp^{\iota^<}\right) = \pi_*\mathcal{O}_X \left( -
\lfloor \iota D \rfloor + K_{X|X_0} + \sum_{E_j\in \Delta\setminus
(S_i\cup[F'_{i+1},F_{i+1}[)} E_j \right).
\end{equation}
In the cases $i<g^*$ or $i=g^* =g-1$ we consider the divisor $E_q$
associated with the dead vertex of $\Gamma$ in the subgraph
$\Gamma_{i+1}$. Applying again Lemma \ref{d}, taking $E_k=E_q$ and
$S=\Delta\setminus (S_i\cup[F'_{i+1},F_{i+1}[)$, we have
\begin{equation}\label{bbb2}
\mathcal{J} \left(\wp^{\iota^<}\right) = \pi_*\mathcal{O}_X \left(
- \lfloor \iota D \rfloor + K_{X|X_0} + \sum E_j \right),
\end{equation}
where $E_j$ runs over the set $\Delta\setminus (S_{i+1}\setminus
\{F_{i+1}\})$.

As a consequence, Equality (\ref{www}) for $i+1$ happens whenever $F_{i+1}\not\in S_{i+1}$.
So we can assume that
$F_{i+1}\in S_{i+1}$ (that is, $\iota$ is a candidate jumping number from $F_{i+1}$
and $F_{i+1}$ does not contribute $\iota$). To conclude our proof we study the three, a priori, existing possibilities.\\

\noindent {\it Case a}. $F_{i+1}$ meets a divisor in
$\Delta\setminus S_{i+1}$ associated with a vertex of the graph
$\Gamma_i\cup \Gamma_{i+1}$. This case cannot happen because then the divisor $F_i$ must contribute $\iota$ and meet
$F_{i+1}$; but then $F'_{i-2}, F'_{i-1}\in \Delta$ (by Lemma
\ref{c}) and this implies that $\lfloor \iota D \rfloor \cdot F_i =D
\cdot F_i=0$, which is a contradiction by Proposition
\ref{CINCO}. \\

\noindent {\it Case b}. $F_{i+1}$ meets a divisor in
$\Delta\setminus S_{i+1}$ associated with a vertex of the graph
$\Gamma_{i+2}$ (this cannot happen if $i=g^*$; so we assume here
that $i<g^*$). Applying Lemma \ref{c}, every divisor meeting
$F_{i+1}$ belongs to $\Delta$ and, therefore, $\lfloor \iota D
\rfloor \cdot F_{i+1}= \iota D\cdot F_{i+1}=0$. Hence
$$\left( \lfloor \iota D \rfloor -K_{X|X_0} -\sum_{E_j\in \Delta\setminus (S_{i+1}\setminus
\{F_{i+1}\}) } E_j \right)\cdot F_{i+1}=1$$ since $K_{X|X_0}\cdot
F_{i+1}=-F_{i+1}^2-2$. Thus, in this case,  by  Lemma
\ref{descarga},
(\ref{www}) holds for $i+1$.\\

\noindent {\it Case c}. $F_{i+1}$ does not meet any divisor in
$\Delta\setminus (S_{i+1}\setminus \{F_{i+1}\})$. Then
$$\left(\lfloor \iota D \rfloor -K_{X|X_0} -\sum_{E_j\in \Delta\setminus (S_{i+1}\setminus
\{F_{i+1}\}) } E_j \right)\cdot F_{i+1}=(\lfloor \iota D \rfloor
-K_{X|X_0}-F_{i+1})\cdot F_{i+1}=$$
$$= \lfloor \iota D \rfloor \cdot F_{i+1}+2>0,  $$
where the inequality holds since, by Proposition \ref{CINCO},
$\lfloor \iota D \rfloor \cdot F_{i+1}\geq -1$. This finishes the
proof after applying Lemma \ref{descarga}. $\Box$

\subsection{Proof of Theorem \ref{UNO}}
We end this paper by proving our main result, Theorem~\ref{UNO},
that provides an explicit expression for the series $P_{\wp} (t)$.

For a start we state the following result, proved in \cite[Lem.
4]{d-g-n-2} in a more general framework, that will be useful.

\begin{theo}
\label{anafel} Let $V = \{ \nu_1, \nu_2, \ldots, \nu_s\}$ be a
finite family of divisorial valuations of $K$ centered at $R$. Set
$S_V :=\left\{ (\nu_1(h), \nu_2(h), \ldots, \nu_s(h)) \in
\mathbb{Z}^s| h \in R \setminus \{0\} \right\}$, $\mathbb{Z}$
denoting the integer numbers and $$B^l := (\nu_1(\psi_l),
\nu_2(\psi_l), \ldots, \nu_s(\psi_l))= (B_1^l,B_2^l,
\ldots,B_s^l),$$ where $1 \leq l \leq s$ and $\psi_l$ is a general
element for $\nu_l$. Then the following statements hold:
\begin{enumerate}
    \item Suppose $s\geq 2$,  fix and index $l$ and consider another one  $1 \leq k \leq
    s$, $k \neq l$; if $\alpha:= (\alpha_1, \alpha_2,\ldots,
    \alpha_s) \in S_V$, then
    \[d_k (\alpha):= \dim_{\mathbb{C}} \frac{\pi_* \mathcal{O}_X (- \sum_{j=1}^s \alpha_j
E_j)}{\pi_* \mathcal{O}_X (- \sum_{j=1}^s \alpha_j E_j- E_k)} \; =
\; \dim_{\mathbb{C}} \frac{\pi_* \mathcal{O}_X (- \sum_{j=1}^s
(\alpha_j+ B^l_j) E_j)}{\pi_* \mathcal{O}_X (- \sum_{j=1}^s
(\alpha_j+ B^l_j) E_j- E_k)}.
    \]
    \item Assume that  $d_l (\alpha) \neq 0$ for some index $l$, then
\[ \dim_{\mathbb{C}} \frac{\pi_* \mathcal{O}_X (- \sum_{j=1}^s (\alpha_j +
B^l_j)E_j)}{\pi_* \mathcal{O}_X (- \sum_{j=1}^s (\alpha_j  + B^l_j)
E_j- E_l)} = 1 \; + \;  \dim_{\mathbb{C}} \frac{\pi_* \mathcal{O}_X
(- \sum_{j=1}^s \alpha_j E_j)}{\pi_* \mathcal{O}_X (- \sum_{j=1}^s
\alpha_j E_j- E_l)}. \;\;\;\Box\]

\end{enumerate}
\end{theo}

Now we return to the situation and notations of this paper, and recall that a divisor with
exceptional support, $E$, is named {\it antinef} whenever $E\cdot
E_j\leq 0$ for all $j=1,2,\ldots,n$. We shall use the following
well-known result (see \cite[Lem. 1.2]{li4} for instance): If $E$
is a divisor on $X$ with exceptional support, then there is an
antinef divisor $E^{-}$ (called {\it antinef closure} of $E$) such
that $E^{-}\geq E$ and $\pi_* \mathcal{O}_X(-E)=\pi_*
\mathcal{O}_X(-E^{-})$ (in fact, $E^{-}$ is the least antinef
divisor $\geq E$). It can be computed by means of the following
procedure: If $E\cdot E_j\leq 0$ for all $j$ then $E^{-}=E$;
otherwise set $E':=E+E_j$, where $E_j$ is such that $E\cdot E_j>
0$, and repeat the procedure replacing $E$ by $E'$. Due to Lemma
\ref{descarga}, the antinef closure of $E$ will be obtained after
finitely many steps.\\

For each $\iota \in \mathcal{H}_i$, $1 \leq i \leq g^* +1$, we
denote by $d^i_\iota$ the following dimension that we shall use in the proof.
$$d^i_\iota := \dim_{\mathbb{C}}\left( \frac{\pi_* \mathcal{O}_X ( K_{X|X_0} - \lfloor \iota D
\rfloor+F_i )}{\mathcal{J} (\wp^\iota)}  \right).$$

By Theorem \ref{TRES} and Lemma \ref{ideafix} we have that
$P_{\wp} (t) = \sum_{i=1}^{g^* +1} P_i (t)$, where $P_i (t) =
\sum_{\iota \in \mathcal{H}_i} d^i_\iota t^\iota$. So, we only need to
compute $P_i (t)$ for any $i$. Write $z_i=t^{1/(e_{i-1}
\bar{\beta}_{i})}$.\\

Firstly assume that $1 \leq i \leq g^*$. To get $ P_i (t)$, we shall use the following

\begin{lem}\label{panoramix}
Let $i$ be an index as above. Then,

\begin{itemize}

\item[(1)] The map $(p,q,s)\mapsto \iota(i,p,q,s)$ gives a bijection
between the sets $B:=\{(p,q,s)\in \mathbb{Z}^3\mid p,q\geq 1,\;\; p
\bar{\beta}_{i} + q e_{i-1}  \leq n_i \bar{\beta}_{i} \mbox{ and } 0
\leq s \leq e_i -1\}$ and ${\mathcal H}_i\cap [0,1]$.

\item[(2)] If $\iota\in {\mathcal H}_i$ and $\iota>1$, then there
exist a unique $(p,q,s)\in B$ and a unique positive integer $r$
such that $\iota=\iota(i,p,q,s)+r=\iota(i,p,q,s+re_i)$.

\end{itemize}
\end{lem}

\begin{proof}
(1) follows from the fact that $e_{i-1}$ and $\bar{\beta}_{i}$
are relatively prime and (2)  from (1) and the
arithmetical expressions of the jumping numbers in ${\mathcal H}_i$.
\end{proof}

As a consequence of the above result we obtain the equality
\[
P_i(t) = \sum_{p,q \geq 1, r \geq 0, p \bar{\beta}_{i} + q e_{i-1}
\leq n_i \bar{\beta}_{i}} d^i_{\iota(i,p,q,r)} z_i^{p
\bar{\beta}_{i} + q e_{i-1} + r n_i \bar{\beta}_{i}}.
\]
Fix any triple of non negative integers $(p,q,s)$ in $B$ and write
\[
\sigma(i,p,q,s) := z_i^{ p \bar{\beta}_{i} + q e_{i-1} + s n_i
\bar{\beta}_{i} } \sum_{r \geq 1} d^i_{\iota(i,p,q,s+re_i)}
z_i^{re_i n_i \bar{\beta}_{i}}.
\]
Then one gets $$P_i(t) = \sum_{(p,q,s) \in B} \sigma(i,p,q,s).$$
So,  we are going to compute the expressions $\sigma(i,p,q,s)$. To do it we shall prove the
following

\begin{lem}\label{abraracurcix}
With the above notations and assumptions, it holds that
\[d^i_{\iota(i,p,q,s+re_i)} = d^i_{\iota(i,p,q,s)}\]
for all non-negative integer $r$.
\end{lem}
\begin{proof}

We shall reason by induction on $r$. Since the equality is evident
for $r=0$, we assume that
$d^i_{\iota(i,p,q,s+re_i)}=d^i_{\iota(i,p,q,s+(r-1)e_i)}$. Set $
\mathcal{J} \left(\wp^{\iota(i,p,q,s+re_i)} \right) = \pi_*
\mathcal{O}_X \left(- \sum_{j=1}^n \alpha_j E_j\right)$, where
\[
(\alpha_1,  \ldots, \alpha_n) = (\lfloor \iota(i,p,q,s+re_i) \nu
(\varphi_1) \rfloor - \kappa_1, \ldots, \lfloor \iota(i,p,q,s+re_i)
\nu (\varphi_n)\rfloor - \kappa_n),
\]
$\sum_{j=1}^n \kappa_j E_j$ being $K_{X|X_0}$. Thus
\[
\mathcal{J} \left(\wp^{\iota(i,p,q,s+(r+1)e_i)} \right) = \pi_*
\mathcal{O}_X \left(- \sum_{j=1}^n (\alpha_j + \nu(\varphi_j)) E_j
\right).
\]

Let $\sum_{j=1}^n \beta_j E_j$ {\LARGE (}$\sum_{j=1}^n (\beta'_j +
\nu(\varphi_j)) E_j $, respectively{\LARGE )} be the antinef closure
of the divisor $\sum_{j=1}^n \alpha_j E_j - F_i$ {\LARGE
(}$\sum_{j=1}^n \left(\alpha_j + \nu(\varphi_j)\right) E_j - F_i$,
respectively{\LARGE )}. As $D \cdot E_j =0$ ($1 \leq j \leq n-1$)
and $D\cdot E_n=-1$, $D$ being the associated to $\wp$ divisor $D =
\sum_{j=1}^n \nu(\varphi_j) E_j$, it is easy to deduce from the
above described procedure for computing antinef closures that
$\beta'_j=\beta_j$ whenever $j<n$ and $\beta'_n\leq \beta_n$.
Moreover one has that $\beta_{st_i}=\beta'_{st_i}=\alpha_{st_i}-1$
(since $F_i$ contributes $\iota$).

Now, consider the commutative diagram

\[
\begin{CD}
\frac{\pi_* \mathcal{O}_X \left(D(\alpha) + F_i\right)}{\pi_*
\mathcal{O}_X \left(D(\alpha)\right)} @= \frac{\pi_* \mathcal{O}_X
\left(D(\alpha) + F_i\right)}{\pi_* \mathcal{O}_X
\left(D(\alpha)\right)} @>f>> \frac{\pi_* \mathcal{O}_X
\left(D(\beta) \right)}{\pi_*
\mathcal{O}_X \left(D(\beta) - F_i \right)}\\
@VVgV @VVV @VVmV \\
\frac{\pi_* \mathcal{O}_X \left(D(\alpha + \varphi)+
F_i\right)}{\pi_* \mathcal{O}_X \left(D(\alpha + \varphi)\right)}
@>h>> \frac{\pi_* \mathcal{O}_X \left(D(\beta' + \varphi)
\right)}{\pi_* \mathcal{O}_X \left(D(\beta' + \varphi)- F_i\right)}
@>i>> \frac{\pi_* \mathcal{O}_X \left(D(\beta + \varphi)
\right)}{\pi_* \mathcal{O}_X \left(D(\beta + \varphi) - F_i\right)}
\end{CD}
\]
\\[2mm]
\noindent where $D(\alpha) := - \sum_{j=1}^n \alpha_j E_j $,
$D(\beta) := - \sum_{j=1}^n \beta_j E_j $, $D(\alpha + \varphi):=
- \sum_{j=1}^n (\alpha_j + \nu(\varphi_j)) E_j $, $D(\beta' +
\varphi):= - \sum_{j=1}^n (\beta'_j + \nu(\varphi_j)) E_j $,
$D(\beta + \varphi):= - \sum_{j=1}^n (\beta_j + \nu(\varphi_j))
E_j $, $f$ and $h$ are the identity homomorphisms (notice that
$\pi_* \mathcal{O}_X \left(D(\alpha)\right)=\pi_*\mathcal{O}_X
\left(D(\beta) - F_i \right)$ and
$\pi_*\left(D(\alpha+\varphi)\right)=\pi_*\mathcal{O}_X
\left(D(\beta'+\varphi) - F_i \right)$), $i$ is defined by the
product by $\varphi_n^{\beta_n-\beta'_n}$ and $g$ and $m$ are
given by  the product by $\varphi_n$ (notice that $\nu(\varphi_j)=
\nu_{E_j} (\varphi_n)$, $\nu_{E_j}$ being the divisorial valuation
 provided by the exceptional divisor $E_j$). This proves
$d_{\iota(i,p,q,s+(r+1)e_i)}=d_{\iota(i,p,q,s+re_i)}$ and thus our
lemma, since $g$ is an isomorphism because $m$ is also an
isomorphism by the proof in \cite[Lem. 4]{d-g-n-2} of Statement (1) in
Theorem $\ref{anafel}$.
\end{proof}

Notice that $d^i_{\iota(i,p,q,s)} =1$ taking into account that
$\iota(i,p,q,s)<1$ ($\iota(i,p,q,s)\not=1$ by \cite[Prop. 8.9]{ja})
and the remark after Lemma \ref{asterix}. Therefore
\[
\sigma(i,p,q,s) = z_i^{p \bar{\beta}_{i} + q e_{i-1} + s n_i
\bar{\beta}_{i}}\sum_{r \geq 1}  z_i^{re_i n_i \bar{\beta}_{i}}=
\frac{z_i^{p \bar{\beta}_{i} + q e_{i-1} + (s+e_i) n_i
\bar{\beta}_{i}}}{1-z_i^{e_i n_i \bar{\beta}_{i}}}
\]
and this implies that
\[
P_i(t) = \frac{1}{1-t} \sum_{\iota \in \mathcal{H}_i, \iota < 1}
t^\iota.
\]

Now we shall obtain an expression for the series $P_{g^*+1}(t)$.
We shall use a similar result to Lemma \ref{panoramix} whose proof
follows from the fact that $e_{g^*}$ and $\bar{\beta}_{g^*+1}$ are
relatively prime.

\begin{lem}\label{wasp}

The following statements hold:

\begin{itemize}

\item[(1)] The map $(s,q)\mapsto \iota(g^*+1,s,q)$ gives a bijection
between the sets $T:=\{(s,q)\in \mathbb{Z}^2\mid 1\leq s\leq e_{g^*}
\mbox{ and } 1\leq q \leq \bar{\beta}_{g^*+1}\}$ and
$\Omega:=\{\iota\in \mathcal{H}_{g^* +1}\mid  \iota\leq 2 \mbox{ and
} \iota-1\not\in \mathcal{H}_{g^* +1} \}$.

\item[(2)] If $\iota\in {\mathcal H}_{g^*+1}$ and $\iota>2$ then
there exist a unique $(s,q)\in T$ and a unique positive integer $r$
such that $\iota=\iota(g^*+1,s,q)+r=\iota(g^*+1,s+re_{g^*},q)$. $\Box$

\end{itemize}

\end{lem}

Set $T$ as in Lemma \ref{wasp}. As a consecuence of that lemma, one can see that
$$P_{g^*+1}(t)=\sum_{(s,q)\in T} \tau(s,q),$$ where
$$\tau(s,q):=\sum_{r\geq 0} d^{g^*+1}_{\iota(g^*+1,s+ r e_{g^*},q)} z_{g^*+1}^{(s+ r e_{g^*})
\bar{\beta}_{g^*+1} + qe_{g^*} }.$$

Applying (2) in Theorem \ref{anafel}, for any fixed pair
$(s,q) \in T$, it happens that
\[
d^{g^*+1}_{\iota(g^*+1,s+re_{g^*},q)} = d^{g^*+1}_{\iota(g^*+1,s,q)}
+ r \mbox{\;\,\; for all $r\geq 0$}.
\]
From this fact, one can deduce the equality
\[
\tau(s,q) = z_{g^*+1}^{ s \bar{\beta}_{g^*+1} +qe_{g^*}} \left[
\frac{d^{g^*+1}_{\iota(g^*+1,s ,q)}}{1 - z_{g^*+1}^{ e_{g^*}
\bar{\beta}_{g^*+1} }} \; + \; \frac{z_{g^*+1}^{ e_{g^*}
\bar{\beta}_{g^*+1} }}{(1 - z_{g^*+1}^{ e_{g^*} \bar{\beta}_{g^*+1}
})^2} \right],
\]
therefore
\[
P_{g^*+1}(t) = \sum_{\iota \in \Omega} t^\iota \left(
\frac{d^{g^*+1}_\iota}{1-t} + \frac{t}{(1-t)^2} \right),
\]
$\Omega$ being as in Lemma \ref{wasp}.

Since $d^{g^*+1}_{\iota}=1$ whenever $\iota<1$ (by the remark after
Lemma \ref{asterix}), it only remains to prove that
$d^{g^*+1}_{\iota}=1$ for all $\iota \in \mathcal{H}_{g^* +1}$ such
that $1<\iota \leq 2$ and $\iota-1$ is not a jumping number (recall
that $1$ is not a jumping number \cite[Prop. 8.9]{ja}). Thus,
consider a jumping number $\iota$ satisfying these conditions and,
reasoning by contradiction, assume that $d^{g^*+1}_{\iota}\geq 2$.
Let $\alpha_j$ be the coefficient of $E_j$ in the divisor $\lfloor
\iota D \rfloor -K_{X|X_0}$, $1\leq j\leq n$. The natural
monomorphism of vector spaces
$$\frac{\pi_*{\mathcal O}_X(-\lfloor \iota D \rfloor +K_{X|X_0}+E_n)}{{\mathcal J}(\wp^{\iota})}\rightarrow \frac{\pi_{*}{\mathcal O}_X(-(\alpha_n-1)E_n)  }{\pi_{*}{\mathcal O}_X(-\alpha_n E_n)}$$
and \cite[Th. 1]{ga} show that the vector space on the left is
generated by classes of elements of the type $\prod_{k=0}^{g+1}
\varphi_{l_k}^{b_k}$, where $b_k$, $0\leq k\leq g+1$, are
nonnegative integers and $l_0,\ldots, l_{g+1}$ are as in the
paragraph after Definition \ref{general}. Since two elements of this
type satisfying the condition $b_{g+1}=0$ are linearly dependent
(see the proof of \cite[Th. 1]{ga}) there exists $f\in R$ such that
the class of $f \varphi_n$ is a non-zero element of the vector space
${\pi_*{\mathcal O}_X(-\lfloor \iota D \rfloor
+K_{X|X_0}+E_n)}/{{\mathcal J}(\wp^{\iota})}$ and therefore
$\nu_{E_n}(f\varphi_n)=\nu(f\varphi_n)=\alpha_n-1$ and
$\nu_{E_j}(f\varphi_n)\geq \alpha_j$, $1\leq j<n$. Thus
$\nu_{E_n}(f)=\alpha_n-1-\nu_{E_n}(\varphi_n)$ and $\nu_{E_j}(f)\geq
\alpha_j-\nu_{E_j}(\varphi_n)$, $1\leq j<n$. This means that $f$ is
in $ \pi_*{\mathcal O}_X(-\lfloor (\iota - 1) D \rfloor
+K_{X|X_0}+E_n)$ but it is not in  $ \pi_*{\mathcal O}_X(-\lfloor
(\iota - 1) D \rfloor +K_{X|X_0})$ what implies that $\iota-1$ is a
jumping number, which contradicts our assumptions. This concludes
the proof of Theorem \ref{UNO}. $\Box$\\

\pushQED{\qed}


\begin{thebibliography}{99}
\footnotesize \setlength{\baselineskip}{3mm}
\bibitem{bu} N. Budur, On Hodge
spectrum and multiplier ideals, {\em Math. Ann.} {\bf 327} (2003),
257---270.
\bibitem{cam} A. Campillo, ``Algebroid curves in positive characteristic",
Lecture Notes in Math. 613. Springer-Verlag (1980).
\bibitem{d} F. Delgado, The semigroup of values of a curve singularity with
several branches, {\em Manuscripta Math.} {\bf 59} (1987),
347---374.
\bibitem{d-g-n} F. Delgado, C. Galindo, A. Nu\~nez, Saturation for valuations
on two-dimensional regular local rings, {\em Math. Z.} {\bf 234}
(2000), 519---550.
\bibitem{d-g-n-2}  F. Delgado, C. Galindo, A. Nu\~nez, Generating
sequences and Poincar\'e series for a finite set of plane divisorial
valuations, to appear in {\it Adv. Math.}
\bibitem{e-l-s-v}  L. Ein, R. Lazarsfeld, K.E. Smith, D. Varolin,
Jumping coefficients of multiplier ideals, {\it Duke Math. J.} {\bf
123} (2004), 469---506.
\bibitem{e-f} F. Enriques, O. Chisini, ``Lezioni sulla teoria geometrica delle
equazione e delle funzioni algebriche" (1915) (Collana di
Matematica, 5. Bologna: N. Zanichelli 1985).
\bibitem{fav} C. Favre, M. Jonsson, Valuations and multiplier
ideals, {\em J. Amer. Math. Soc.} {\bf 18} (2005), no. 3, 655---684
(electronic).
\bibitem{ga} C. Galindo, On the Poincar\'e  series for a plane divisorial valuation,
{\em Bull. Belg. Math. Soc.} {\bf 2 } (1995), 65---74.
\bibitem{ho} J. Howald, Multiplier ideals of monomial ideals,
 {\em Trans. Amer. Math. Soc.} {\bf 353} (2001), 2665---2671.
\bibitem{ja} T. J\"{a}rvilehto,``Jumping numbers of a simple
complete ideal in a two-dimensional regular local ring", Ph. D.
thesis, University of Helsinky  (2007).
\bibitem{la} R. Lazarsfeld,
``Positivity in algebraic geometry. Vol. II",  Springer (2004).
\bibitem{li0} J. Lipman, Rational singularities, with applications
to algebraic surfaces and unique factorization, {\it Inst. Hautes
\'Etudes Sci. Publ. Math.} {\bf 36} (1969), 195---279.
\bibitem{li1} J. Lipman, ``On complete ideals in regular local rings" in Algebraic geometry
and commutative algebra in honor of M. Nagata (1987), 203---231.
\bibitem{li2} J. Lipman, Adjoints of ideals in regular local rings, {\it Math. Res. Lett.}
{\bf 1} (1994), 793---755. With an appendix by S.D. Cutkosky.
\bibitem{li4} J. Lipman, K. Watanabe, Integrally closed ideals in
two-dimensional regular local rings are multiplier ideals, {\it
Math. Res. Lett.} {\bf 10} (2003), 423---434.
\bibitem{ni-zu} I. Niven, H. Zuckerman, ``An introduction to the
theory of numbers", John Wiley \& Sons (1972).
\bibitem{sai} M. Saito, Exponents of an irreducible plane curve
singularity, arXiv:math/0009133v2.
\bibitem{sai2} M. Saito, On Steenbrink's Conjecture, {\it Math. Ann.}
{\bf 289} (1991), 703---716.
\bibitem{s-t} K.E. Smith, H.M. Thompson, Irrelevant exceptional divisors for curves
on a smooth surface, to appear in the volume of the  Midwest
Algebra, Geometry and their Interactions Conference 2005.
\bibitem{spi} M. Spivakovsky, Valuations  in  function   fields   of
surfaces, {\em Amer. J. Math.} {\bf 112} (1990), 107---156.
\bibitem{ste1} J.H.M. Steenbrink, ``Mixed Hodge structures on the vanishing
cohomology" in Real and Complex Singularitites, Oslo (1976),
Alphen aan den Rijn, Oslo, (1977), 525---563.
\bibitem{ste2} J.H.M. Steenbrink, The spectrum of hypersurface
singularities, {\it Ast\'erisque},
{\bf 179-180} (1989), 163--184.
\bibitem{tuc} K. Tucker, Jumping numbers on algebraic surfaces with
rational singularities, arXiv:math/0801.0734v2.
\bibitem{var} A. N. Varchenko, Assymptotic Hodge structure in the
vanishing cohomology, {\it Math USSR Izv.} {\bf 18} (1982),
469---512.
\bibitem{z-3} O. Zariski,``Algebraic surfaces", 2nd suppl. ed.,
 Ergebnisse 61,  Springer Verlag (1971).
\bibitem{z-1} O. Zariski, Polynomial ideals defined by infinitely near base points,
{\em Amer. J. Math.} {\bf 60} (1938), 151---204.
\bibitem{zar}
O. Zariski, P. Samuel, ``Commutative algebra.  Vol.  II",
Springer-Verlag (1960).
\end{thebibliography}
\end{document}